\documentclass{birkjour}
\newcommand{\cl}{\operatorname{\mathcal{C}\ell}}
\newcommand{\Pol}{\operatorname{\mathsf{Pol}}}
\newcommand{\R}{\mathbb{R}}
\newcommand{\C}{\mathbb{C}}
\newcommand{\Z}{\mathbb{Z}}
\newcommand{\mE}{\mathbb{E}}
\newcommand{\mS}{\mathbb{S}}
\newcommand{\mcS}{\mathcal{S}}
\newcommand{\sym}{\mathfrak{sp}}
\newcommand{\spl}{\mathfrak{sl}}
\newcommand{\so}{\mathfrak{so}}
\newcommand{\ux}{\underline{x}}
\newcommand{\uy}{\underline{y}}

\newcommand{\up}{\underline{\partial}}
 \newtheorem{thm}{Theorem}[section]
 \newtheorem{cor}[thm]{Corollary}
 \newtheorem{lem}[thm]{Lemma}

 \theoremstyle{definition}
 \newtheorem{defn}[thm]{Definition}
 \theoremstyle{remark}
 \newtheorem{rem}[thm]{Remark}
  \newtheorem{notation}[thm]{Notation}
 
 \numberwithin{equation}{section}

\usepackage{tikz}
\usepackage{tikz-cd}
\begin{document}

%
%
%
%
%
%
%
%
%

\title[The Symplectic Fueter-Sce Theorem]
 {The Symplectic Fueter-Sce Theorem}

\author{David Eelbode}
\email{david.eelbode@uantwerpen.be}

\author{Sonja Hohloch}
\email{sonja.hohloch@uantwerpen.be}

\author{Guner Muarem}
\email{guner.muarem@uantwerpen.be}

\address{\\ University of Antwerp \\ Department of Mathematics\\ Middelheim Campus - Building G \\
Middelheimlaan 1\\ 2020 Antwerpen - Belgium}


\date{August 1, 2019}

\begin{abstract}
In this paper we present a symplectic analogue of the Fueter theorem. This allows the construction of special (polynomial) solutions for the symplectic Dirac operator $D_s$, which is defined as the first-order $\sym(2n)$-invariant differential operator acting on functions on $\R^{2n}$ taking values in the metaplectic spinor representation. 
\end{abstract}

\maketitle

\section{Introduction}
Fueter's original theorem in quaternionic analysis (see \cite{F}) is a result which allows the construction of so-called regular functions (solutions for a generalised Cauchy-Riemann operator $D = \partial_{x_0} + i\partial_{x_1} + j\partial_{x_2} + k\partial_{x_3}$ acting on smooth $\mathbb{H}$-valued functions), starting from holomorphic functions $f(z)$ in the complex plane $\C$. In sharp contrast to the space of holomorphic functions, the space containing regular functions does not define an algebra; therefore, having a systematic way to construct null solutions is clearly a powerful tool. The Fueter theorem does just that, and for that reason this result has later been generalized within the setting of Clifford analysis, a branch of analysis in which null solutions for spin-invariant differential operators are studied. Without claiming completeness, we refer to e.g. \cite{Sc,KQS, PPS, Q, S}. Note that the classical Fueter theorem was generalised by Sce \cite{Sc} to the Clifford setting. This result is now know as the Fueter-Sce theorem.
\\ 
\noindent
The main object of study in this subdomain of classical analysis is the Dirac operator $\up_x := \sum_j e_j \partial_{x_j}$, with $\{e_1,\ldots,e_m\}$ an orthonormal basis for $\R^m$ (the uninitiated reader is referred to the standard books \cite{BDS, DSS, GM}). This first-order operator is the unique (up to normalisation) conformally invariant operator acting on smooth spinor-valued functions $f(\ux) \in \mathcal{C}^\infty(\R^m,\mS)$ on the Euclidean space $\R^m$. The spinor space $\mathbb{S}$ can hereby be realised as a minimal left ideal in the complex Clifford algebra ${\mathcal{C}\ell(\C^m)}$ {(denoted by $\C_m$ from now on)}, defined as the associative algebra generated by the orthonormal basis $\{e_1,\ldots,e_m\}$ endowed with the multiplication rules $e_pe_q + e_qe_p = -2\delta_{pq}$ (the minus sign is a convention here, so that one ends up with several anti-commuting complex units).
 In even dimensions $m = 2n$ (which, as we will see in {S}ection 2, is enough for our purposes), this can be done by means of a suitable multiplication from the right with a primitive idempotent $I$ (see e.g. \cite{DSS}). Indeed, one has that $\mS := \C_m I$. This means that one can safely work in the Clifford algebra $\C_m$, and multiply everything with $I$ at the end of the calculations. 
{We further} note that the Clifford multiplication rules also lead to the crucial relation $-\up_x^2 = \Delta_x$ (the Laplace operator in $m$ variables), which explains why Clifford analysis is sometimes referred to as a `refinement' of harmonic analysis on $\R^m$. \\  
{ One of the main goals of Clifford analysis is to obtain solutions of the Dirac operator $\up_x$ and study their properties. In order to avoid problems with the topology of underlying functions spaces, we will always restrict ourselves to polynomial solutions of the operator $\up_x$. We denote the set of all Clifford algebra-valued monogenic polynomials of degree $k\in\mathbb{N}$ by $\mathcal{M}_k(\R^m,\mathbb{S})$. In other words, we have that \begin{align*}
	\mathcal{M}_k(\R^m,\mathbb{S}) = \Pol_k(\R^m,\mathbb{S}) \cap \ker \up_x.
	\end{align*}
}
\noindent
The Fueter theorem on $\R^m$ (together with its various generalisations) then yields a method to construct { these null-solutions} starting from arbitrary holomorphic functions $f(z)$. These solutions, which can be expressed in terms of the Gegenbauer polynomials, play an important role in the representation theory for the spin group Spin$(m)$, in particular within the setting of branching rules and axially monogenic polynomials on $\R^m$. This connection between the Fueter theorem and the representation theory for Lie algebras became even more apparent in \cite{ESVL}: in this paper, the construction of the special Fueter solutions for the operator $\up_x$ was obtained as an intertwining operator between certain lowest-weight modules for the Lie algebra ${\spl(2)}$. This was based on the idea that the Lie algebra ${\spl(2)}$ can be realised as a subalgebra of the conformal Lie algebra $\mathfrak{so}(1,m+1)$, which contains all the (first-order) generalised symmetries for the Dirac operator. The main ideas behind this construction can be found in Section 2.  \\
\noindent
{ We further emphasise that} the Fueter operator $D$ also appears in other domains of mathematics. For instance, the differential operator $\underline{D} := i\partial_{x_1} + j\partial_{x_2} + k\partial_{x_3}$ (without the term in $x_0$) plus a zero order Hamiltonian term $H_0$ can be used to describe the critical points of a hypersymplectic action functional. If one then reintroduces the coordinate $x_0$ as a time variable, the non-homogeneous Fueter equation $(D + H_0)f = 0$ can be seen as the $L^2$-gradient flow equation associated to the hypersymplectic action functional, viewed as a Morse function. The flow lines are the so-called connecting Floer cylinders in Hyperk\"ahler Floer theory: this homology can be calculated in certain situations as the Morse homology of the underlying Hyperk\"ahler manifold. Via Arnold's conjecture, this can then be used to estimate the number of solutions of the equation $(\underline{D} + H_0)f = 0$ mentioned above by the sum of the Betty numbers of the manifold (see the paper \cite{HNS} by Hohloch et al.). The operator $\underline{D}$ also appeared in the work \cite{W} of Walpuski, where compactness issues were studied, and Haydys' thesis \cite{Ha}, in connection with tri-holomorphicity and the Seiberg-Witten equations. \\
\noindent
The aim of the present paper is to generalise the approach from \cite{ESVL}, for the specific case of the symplectic Dirac operator $D_s$ which was introduced and studied in \cite{DBHS, DBSS}.
 This is a first-order differential operator, acting on functions $f(\ux,\uy)$ on $\R^{2n}$ taking values in the spinor representation for the symplectic Lie algebra $\sym(2n)$, which commutes with the regular action of said Lie algebra (see Section 3 for more details). Using the fact that one can also construct generalised symmetries for the operator $D_s$, we will be able to formulate a symplectic version of the Fueter theorem, which will then lead to a special collection of solutions for the operator $D_s$.
 { Note that this are the first steps towards a better understanding of the representation theoretical aspects of the space $\ker D_s$. Furthermore, it allows us to obtain results concerning the interaction with special functions and to comprehend how these functions change when the underlying geometry of the space is altered (i.e. from an orthogonal to a symplectic framework).}
 { One of the main} differences with the orthogonal case (centered around the Dirac operator and the spin group) will be the fact that one cannot start from holomorphic functions $f(z)$; instead one has to consider other special functions as a starting point. This will be shown in Section 4. 
{In the future we would like to better understand the symplectic Fueter mapping $\mathcal{F}_s$ by introducing an inverse Fueter mapping theorem in the context of \cite{CSS}.}
\begin{notation}
{We fix some elementary notations for the rest of this article. \begin{itemize}
	\item First of all we note that every Lie algebra is assumed to be the complex version, unless otherwise stated by using a subindex. This will be in particular the case for the conformal algebra $\so_{\mathbb{R}}(1,m+1)$.
	\item When writing $\mathbb{N}_0$, $\mathbb{R}_0$, etc. we \textit{exclude} the number $0$ from the set.
\end{itemize}}
\end{notation}

\section{The classical Fueter theorem}

To illustrate the method we will use to construct special solutions for the symplectic Dirac operator $D_s$, we will first briefly consider the orthogonal case again.  

As mentioned earlier, one possible approach to prove Fueter's theorem for the Dirac operator $\up_x$ is based on the existence of a particular subalgebra ${\spl(2)} \subset \mathfrak{so}(1,m+1)$ of the full conformal Lie algebra containing the (first-order) generalised symmetries. Introducing $\mathcal{W}(\R^m) \otimes \C_m$ as the tensor product of the Weyl algebra (with $2m$ generators $x_j$ and $\partial_{x_j}$) and the Clifford algebra, these are defined as follows: 
\begin{defn}
An operator $D$ of the special form 
\[ {D := \sum_{j = 1}^m p_j(x_k,{\partial_x}_k) \otimes c_j\in \mathcal{W}(\R^m) \otimes \C_m\ ,} \] 
 is called a generalised symmetry for the Dirac operator $\up_x$ if it satisfies the following criterion: 
\[ \big(\exists D' \in \mathcal{W}(\R^m) \otimes \C_m : \up_x D = D'\up_x\big)\ . \]
As such, one clearly has that $D \in \mathsf{End}(\ker \up_x)$. In the specific case where $D' = D$, one says that $D$ is a proper symmetry (with $[D,\up_x] = 0$).
\end{defn}
\noindent
Easy examples of proper symmetries are the generators $\partial_{x_j}$ of translations in the $e_j$-direction, and the generators 
\[ M_{ij} := x_i\partial_{x_j} - x_j\partial_{x_i} - \frac{1}{2}e_{ij} \in \mathcal{W}(\R^m) \otimes \C_m \qquad (1 \leq i < j \leq m) \] 
for the regular action of the Lie algebra $\so(m)$ on spinor-valued functions $f(\ux)$. However, in order to prove Fueter's theorem using algebraic methods, one needs the so-called special conformal transformations. These are typically defined in terms of the Kelvin inversion operator: 
\begin{defn}
For arbitrary $\alpha\in \R$, the $\alpha$-deformed inversion operator on $\C_m$-valued functions $f(\ux)$ is defined by 
\begin{align}
\mathcal{I}_{\alpha}^{(m)}[f](\ux):=\frac{\ux}{|\ux|^{m+\alpha}}f\left(\frac{\ux}{|\ux|^2}\right).
\end{align}
\end{defn}
\noindent
For $\alpha = 0$, this operator reduces to the aforementioned Kelvin inversion: 
\[ f(\ux) \in \ker(\up_x)\ \Rightarrow\ \mathcal{I}_{0}^{(m)}[f](\ux) = \frac{\ux}{|\ux|^{m}}f\left(\frac{\ux}{|\ux|^2}\right) \in \ker(\up_x)\ . \]
Put differently, one has $\mathcal{I}_{0}^{(m)} \in \mathsf{End}(\ker\up_x)$. It might be useful to point out here that the domain of the original function $f(\ux)$ and its image under the Kelvin inversion are not necessarily the same (polynomials are for instance mapped to rational functions {when $\alpha$ is a positive integer}). What is even more important for what follows, is the following observation (we restrict our attention to polynomials below, but this will turn out to be sufficient for our purposes): 
\begin{lem}
Let $\alpha = -2k$ with $k \in \Z^+$, and let $P(\ux)$ be a polynomial on $\R^m$ with values in $\C_m$, such that $\up_x^{2k+1} P(\ux) = 0$. One then has that 
\[ \mathcal{I}_{-2k}^{(m)}[P](\ux) \in \ker(\up_x^{2k+1})\ . \]
\end{lem}
\begin{proof} Without loss of generality, we may consider a homogeneous polynomial, say of degree $\ell \in \Z^+$. We can then decompose $P(\ux)$ into a sum of monogenic polynomials (this is the well-known Fischer decomposition, see for instance \cite{DSS}). Given the fact that $\up_x^{2k+1} P(\ux) = 0$, one gets at most $(2k+1)$ terms here: 
\[ P(\ux) = M_\ell(\ux) + \ux M_{\ell-1}(\ux) + \ux^2 M_{\ell-2}(\ux) + \ldots + \ux^{2k}M_{\ell - 2k}(\ux)\ , \]
where {$M_i(\ux)\in\mathcal{M}_i(\R^m,\mathbb{S})$  with $i$ the appropriate degree of homogeneity and} $M_{\ell - j}(\ux) = 0$ if the corresponding index $j > \ell$. It is then clear that 
\begin{align*}
\mathcal{I}_{-2k}^{(m)}[\ux^a M_{\ell-a}](\ux) &= \frac{\ux}{|\ux|^{m-2k}}\left(\frac{\ux}{|\ux|^2}\right)^aM_{\ell-a}\left(\frac{\ux}{|\ux|^2}\right)\\
&= \ux^a \mathcal{I}_{0}^{(m)}[M_{\ell-a}](\ux)\ ,
\end{align*}
with $0 \leq a \leq 2k$. This means that the action of $\up_x^{2k+1}$ will give zero. 
\end{proof}
\noindent
Despite the fact that the ($\alpha$-deformed) inversion does not preserve the space of polynomials, the conjugation of a partial derivative $\partial_{x_j}$ with this inversion does give an operator which behaves nicely on ($\C_m$-valued) polynomials. In what follows we will fix $j = 1$, which means that the $e_1$-direction plays a preferential role from now on. This is also reflected in the notation, since we will from now on write 
\[ \ux = e_1x_1 + \sum_{j = 2}^m e_j x_j = e_1 x_1 + \boldsymbol{x} \in \R \oplus \R^{m-1}\ {.} \]
	{ Note that we use the notation $\boldsymbol{x}$ to stress the fact that we take an element in $\R^{m-1}$ (in contrast with an vector in the full space $\ux\in \R^m$) where we cancelled out the preferred $e_1$-direction.}
\begin{rem}
The fact that one singles out a preferential direction (both here and in the symplectic case, see below) has its repercussions on the level of representation theory: it means that the behaviour of the Fueter solutions for the Dirac operator $\up_x$ will have some sort of invariance property with respect to the subalgebra $\so(m-1) \subset \so(m)$. 
\end{rem}
\noindent
For the following result, we refer to \cite{ESVL2}. The operator $\mathbb{E}$ appearing below is defined as the so-called Euler operator $\sum_j x_j\partial_{x_j} = r\partial_r$ (with $r = |\ux|$), which acts diagonally on homogeneous functions {(it measures the degree of homogeneity)}. 
\begin{thm}\label{raising_op}
For an arbitrary real parameter $\alpha$, the raising operator $R^{[m,\alpha]}$ in $\mathcal{W}(\R^m) \otimes \C_m$ is defined by means of
\begin{align}\label{con}
R^{[m,\alpha]} &:= \mathcal{I}_{\alpha}^{(m)}\partial_{x_1}\mathcal{I}_{\alpha}^{(m)}\nonumber\\ 
&= -|\ux|^{2} \partial_{x_1} +x_1( 2\mathbb{E} +m-1+\alpha) +\boldsymbol{x} e_1\ .
\end{align}
{For} $\alpha=-2k$, {with $k\in\mathbb{N}$} this operator plays a special role: 
\[ R^{[m,-2k]} = \mathcal{I}_{-2k}^{(m)}\partial_{x_1}\mathcal{I}_{-2k}^{(m)} \in \mathsf{End}(\ker\up_x^{2k+1})\ . \]
\end{thm}
\noindent
In particular, when $k = 0$ we observe that $R^{[m,0]} \in \mathsf{End}(\ker\up_x)$. As a matter of fact, this is one of the special conformal transformations mentioned earlier (there are $m$ independent ones, obtained by starting from different partial derivatives $\partial_{x_j}$). Note that it is a {\em raising} operator (or a ladder operator), in the sense that its action on a homogeneous polynomial raises the degree by one. It can even be supplemented with a lowering operator, which then leads to a realisation of the Lie algebra ${\spl(2)}$ (see \cite{ESVL2}): 
\begin{align*}
{\spl(2)} = \mathsf{Alg}(X,Y,H) &\cong \mathsf{Alg}\big(R^{[m,\alpha]},-\partial_{x_1},2\mathbb{E} +m-1+\alpha\big)\ .
\end{align*}
The idea behind Fueter's theorem is that one can use the $\alpha$-deformed raising operator $R^{[m,\alpha]}$ to create polynomials which look like holomorphic powers, provided one chooses the appropriate index $\alpha$. To explain this, we first of all note that the Clifford product 
\[ \overline{e}_1\ux = \overline{e}_1(e_1x_1 + \boldsymbol{x}) = x_1 + \overline{e}_1\boldsymbol{x} \in \mathbb{R}^{(0)}_m\oplus \mathbb{R}_m^{(2)} \]
formally behaves like a complex number (the bar hereby denotes the Clifford conjugation, defined on generators by means of $\overline{e}_j=-e_j$). Indeed, putting $\boldsymbol{x} = |\boldsymbol{x}|\boldsymbol{\omega}$ with $\boldsymbol{\omega} \in S^{m-2}$ a unit vector in $\R^{m-1}$, the previous line can be rewritten as $\overline{e}_1\ux = x_1 + (\overline{e}_m\boldsymbol{\omega})|\boldsymbol{x}|$, with $(\overline{e}_m \boldsymbol{\omega})^2 = -1$. This then inspires the following substitution rule: 
\begin{align}
z = \overline{e}_1(e_1x_1 + e_2x_2) \in \mathbb{C}\mapsto \overline{e}_1\ux = x_1 + \overline{e}_1\boldsymbol{x} \in \mathbb{R}^{(0)}_m\oplus \mathbb{R}_m^{(2)}\ ,\label{sub}
\end{align}
where $z = x_1 + \overline{e}_1e_2 x_2$ is just another way to write the more common $x + iy$ in $\C$. This formal substitution can be used to turn holomorphic functions $f(z)$ into a $\C_m$-valued function on $\R^m$: starting from the Taylor expansion for $f(z)$ around the origin, it suffices to replace every holomorphic power $z^k$ by the Clifford power $(\overline{e}_1 \ux)^k$. The crux of the argument is encoded in the following result (see also \cite{ESVL2}), which says that the image of these holomorphic powers can also be obtained in terms of a deformed raising operator: 
\begin{thm}
For the special value $\alpha = 2-m$, one finds that 
\begin{align}
R^{[m,2-m]}(\overline{e}_1 \ux)^k=(k+1)(\overline{e}_1 \ux)^{k+1}.
\end{align}
\end{thm}
\noindent
Fueter's theorem then easily follows from the observation that in case the dimension $m = 2\mu \in 2\Z^+$ is even, one can invoke Theorem \ref{raising_op} from above: 
\[ R^{[m,2-m]} = R^{[m,-2(\mu-1)]} \in \mathsf{End}(\ker\up_x^{2\mu - 1})\ . \]
Since $\up_x^{2\mu - 1} = (-1)^{\mu-1}\Delta_x^{\mu-1}\up_x$, this then means that $\Delta_x^{\mu - 1}f(\overline{e}_1 \ux) \in \ker(\up_x)$, with $f(z)$ a holomorphic function which has a Taylor expansion around the origin $z = 0 \in \C$. 
\begin{rem}
One might be tempted to argue here that $(\overline{e}_1 \ux)^\ell$ is not a spinor-valued function, but as explained in the introduction it suffices to multiply this $\C_m$-valued polynomial with the idempotent $I \in \C_m$ (or any other spinor, for that matter) from the right to remedy this. 
\end{rem}
\noindent
Since we will generalise this approach in the next sections, we list the crucial ideas behind this (algebraic) proof, so that we can use it as a guideline: 
\begin{itemize}
\item[(i)] Formally relate the undeformed raising operator in $m = 2$ to a deformed raising operator in general dimension: 
\[ R^{[2,0]} = R^{[m,2-m]}\ . \]
Note that this is a {\em formal} identification only, because these operators are defined in terms of the Clifford variable $\ux$ whose exact definition depends on the dimension. 
\item[(ii)] Prove that the deformed raising operator has the special property that it still preserves the kernel for a certain operator (which then happens to be a suitable power of the Dirac operator $\up_x$). 
\item[(iii)] Combine both ideas: let the raising operator in the lowest dimension act first (in the orthogonal case, this then gives holomorphic powers), perform the formal substitution ${z = \overline{e}_1(e_1x_1 + e_2x_2)\mapsto \overline{e}_1\ux = x_1 + \overline{e}_1\boldsymbol{x}}$ and act with a suitable power of the Dirac operator to obtain a solution for $\up_x$ defined on $\R^m$. 
\end{itemize}
Given the fact that the Dirac operator $\up_x$ generalises the Cauchy-Riemann operator $\overline{\partial}_z$ in a canonical way, it should come as no surprise that the Fueter theorem relates holomorphic functions to monogenic functions. There is also a second connection, which is situated on the level of special functions. For that purpose, we first of all note that 
\[ z^k = (x + iy)^k \sim \partial_z\big(|z|^{k+1}\cos((k+1)\theta)\big)\ , \]
and this trigonometric function can be seen as a Chebyshev polynomial of the first kind: $T_{k+1}(\cos \theta )=\cos((k+1)\theta)$, where $\Re(z) = |z|\cos\theta$ denotes the real part of $z \in \C$. Moreover, this is then nothing but a degenerate case of a Gegenbauer polynomial $C_k^\alpha(x)$, because it is well-known that 
\[ T_{k+1}(x) = \frac{k+1}{2}\lim_{\mu \rightarrow 0}\frac{1}{\mu}C_{k+1}^\mu(x) \qquad (k \geq 0)\ . \]
On the other hand, the Fueter images of holomorphic powers are (up to a constant) given by the special monogenic polynomials of the form 
\[ M_k(\ux) = \up_x\left(|\ux|^{k+1}C_{k+1}^{\frac{m}{2} - 1}\left(\frac{x_1}{|\ux|}\right)\right)\ , \]
where one can once again say that $x_1 = |\ux|\cos\theta$ (using a suitable spherical coordinate system in $m$ dimensions). This follows from representation theoretical arguments, see for instance \cite{ESVL}.  {  From the abstract branching rules, we have the following decomposition: 	\[ \mathcal{M}_k(\R^m,\mathbb{S}) \bigg|_{\so(m-1)}^{\so(m)}=\bigoplus_{j=0}^k\mathcal{M}_j(\R^{m-1},\mathbb{S})
	\]
	As $m=2n$ in this paper, we note that the spinor space $\mathbb{S}$ on the left-hand side should have a sign $\pm$.
The realisation of the trivial component $j=0$ is then exactly this special function from above.
}

From this point of view, the Fueter theorem says that two families of special functions (indexed by a positive integer $k \in \Z^+_0$) are related by a parameter defined in terms of the dimension of the space one is working with. 
\section{Symplectic Clifford analysis}

In this section, we will prove that a similar scheme arises in the symplectic setting: special solutions for the symplectic Dirac operator $D_s$ (see below) can be constructed in terms of solutions in the lowest possible dimension (which is again $m = 2$). This will again be done in terms of deformed raising operators, and it will lead to two related families of special functions. \\
\\
Let $(\R^{2n},\omega)$ be a symplectic vector space with coordinates 
\[ (\ux,\uy) := (x_1,\ldots,x_n,y_1,\ldots,y_n) \in \R^n \oplus \R^n \cong \R^{2n} \] 
and their associated partial derivatives $(\partial_{x_1},\ldots,\partial_{y_n})$, and where the skew bilinear form $\omega$ is defined by means of $\omega(\underline{u},\underline{v}) := \langle \underline{u},J(\underline{v}) \rangle$ with 
\[ J := \left(\begin{array}{cc}
0 & +\textup{Id}_n\\
-\textup{Id}_n & 0
\end{array}\right) \in \R^{2n \times 2n}\ . \]
We also need the symplectic version of the Clifford algebra {$\C_m$ used in the orthogonal case}.
\begin{defn}
Let $(V,\omega)$ be a symplectic vector space of dimension $2n$. The symplectic Clifford algebra $\cl(V,\omega)$ is defined as the quotient algebra of the tensor algebra $T(V)$ of $V$, by the two-sided ideal $\mathcal{I}_{\omega}$ generated by elements of the form $\{v\otimes u-u\otimes v+\omega(v,u)\mid u,v\in V\}$. In order words
\[ \cl(V,\omega):=T(V)/\mathcal{I}_{\omega} \] 
is the algebra generated by $V$ in terms of the relation $vu-uv=-\omega(v,u)$.
\end{defn}
\noindent
This algebra is isomorphic to the Weyl algebra $\mathcal{W}(\R^n)$, so in sharp contrast to the orthogonal case, the symplectic Clifford algebra has infinite dimension. This algebra has a natural action on the so-called symplectic spinor space, an infinite-dimensional irreducible representation which can be modelled by the Schwartz space $\mcS(\R^n)$, also known as the Segal-Shale-Weil representation (see also the remark below). The generators of the symplectic Clifford algebra then act on this module as multiplication operators $q_j$ and derivatives $\partial_{q_j}$: 
\begin{align*}
e_j \cdot \phi := iq_j\phi\quad \text{and} \quad e_{n+j}\cdot \phi := \partial_{q_j}\phi\ ,
\end{align*}
where $1 \leq j \leq n$ and $\phi \in \mcS(\R^n)$. The role of the orthogonal Lie algebra $\so(m)$ appearing in the classical case is now played by the symplectic Lie algebra $\mathfrak{sp}(2n)$. It has a natural regular action on the space of polynomial symplectic spinors $\Pol(\R^{2n},\C)\otimes \mathcal{S}(\R^n)$. Indeed, the generators for this representation are given by (see e.g. \cite{DBHS}):
 \begin{align}\label{realisaties}
\begin{cases}
X_{jk}&=-x_j\partial_{x_k}+y_k\partial_{y_j} + q_k\partial_{q_j}+\frac{1}{2}\delta_{jk}
\\Y_{jk}&=x_k\partial_{y_j}+x_j\partial_{y_k} + i\partial_{q_j}\partial_{q_k}
\\ Z_{jk}&=y_k\partial_{x_j}+y_j\partial_{x_k} + iq_jq_k
\\Y_{jj}&=-x_j \partial_{y_j} -\frac{i}{2}\partial_{q_j}^2
\\Z_{jj}&=-y_j\partial_{x_j}-\frac{i}{2}q_j^2
\end{cases}
\end{align} 
\begin{rem}
Note that the Schwartz space $\mcS(\R^n)$ decomposes into a direct sum of two unitary representations for the metaplectic Lie group $\mathsf{Mp}(2n,\mathbb{R})$, the double cover for the (real) symplectic Lie group $\mathsf{Sp}(2n,\R)$. This amounts to considering the even or odd functions $\psi(q_1,\ldots,q_n) \in \mcS(\R^n)$. See e.g. \cite{DBSS} for more information. 
\end{rem}
\noindent
One the most important results in (classical) Clifford analysis is the existence of a so-called Howe dual pair, consisting of the spin group (or its associated orthogonal Lie algebra) and a dual symmetry algebra which is generated by the invariant operators. A similar result exists in the symplectic case, and for that purpose we first introduce the natural invariant operators: 
\begin{align*}
X_s&=\sum_{j=1}^n y_j\partial_{q_j}+ix_jq_j
\\ D_s&=\sum_{j=1}^n iq_j\partial_{y_j}-\partial_{x_j}\partial_{q_j}
\\ \mathbb{E}&=\sum_{j=1}^n x_j\partial_{x_j}+y_j\partial_{y_j}.
\end{align*}
Note that the second operator is known as the symplectic Dirac operator, whereas the first operator is its dual. This is the analogue of the multiplication operator $\ux$ from the previous section, but note that the symplectic version is differential in the dummy variable $q$. 
\begin{rem}
The symplectic Dirac operator has also been considered in greater generality, on symplectic manifolds $(M,\omega)$ with an $Mp^c$-structure, by Cahen and Gutt et al.\ in for instance \cite{CH1, CH2, CH3}. For a more recent reference, we mention the work of Nita \cite{Ni}. 
\end{rem}
\noindent
In \cite{DBSS} the well-known Howe duality for the classical Dirac operator on $\R^m$ was generalised to the symplectic context, in such a way that the space of symplectic polynomial spinors decomposes into a direct sum of irreducible subspaces. To do so, the symplectic action defined above was supplemented by the dual symmetry algebra ${\spl(2)}$:
\[ {\spl(2)} = \mathsf{Alg}(X,Y,H) \cong \mathsf{Alg}\left(X_s,D_s,\mathbb{E} + n\right)\ . \] 
Indeed, one easily verifies that the following relations hold (note that these still need to be normalised to obtain the proper ${\spl(2)}$-relations): 
\[ [\mathbb{E} + n,X_s] = +X_s \qquad [\mathbb{E} + n,D_s] = -D_s \qquad [D_s,X_s] = -i(\mathbb{E} + n)\ . \]
Using the dual pair $\mathsf{Mp}(2n,\mathbb{R})\times \mathfrak{sl}(2,\mathbb{R})$, where $\mathsf{Mp}(2n,\mathbb{R})$ denotes the metaplectic group (the double cover for the symplectic group), one can decompose the space of polynomial symplectic spinors $\Pol(\R^{2n},\C)\otimes \mathcal{S}(\R^n)$ in terms of the following building blocks: 
\begin{defn} 
For arbitrary $\ell \in\Z^+$, the space of $\ell$-homogeneous symplectic (polynomial) monogenics is defined as
\[
\mathcal{M}_{\ell}^s(\R^{2n}) := \ker(D_s)\cap (\mathsf{Pol}_{\ell}(\R^{2n})\otimes\mathcal{S}(\mathbb{R}^n)).
\]
\end{defn}
\noindent
Just as for the space of symplectic spinors, these spaces decompose into two irreducible representations for the action of the metaplectic group (only for $\ell = 0$ they are unitary though). For the following result we refer to \cite{DBSS}: 
\begin{thm}\label{sym_decomp}
Under the action of the $\sym(2n)$, the space of polynomial symplectic spinors decomposes as
\begin{align}
\mathsf{Pol}(\R^{2n})\otimes\mathcal{S}(\mathbb{R}^n) &= \bigoplus_{j = 0}^\infty\bigoplus_{\ell = 0}^\infty X_s^j \mathcal{M}_{\ell}^s(\R^{2n})\ .
\end{align}
\end{thm}


\section{Symplectic Fueter theorem}

In order to derive a symplectic version of the Fueter theorem, we will first consider a special class of (generalised) symmetries for the symplectic Dirac operator $D_s$, which then generalises the special conformal transformations introduced in {S}ection 2. However, as will become clear soon, it is again useful to consider a slight deformation of the true generalised symmetry. Note that these symmetries first appeared in \cite{DBHS}, whereas in \cite{HKS} the authors studied a special case in connection with projective geometry in two dimensions. 
\begin{defn}
Let $m=2n$ and $\alpha \in \R$ a perturbation parameter. Then the symplectic $\alpha$-deformed raising operator is for all $1 \leq j \leq n$ defined as 
\begin{align} \label{raisingsymplectic}
Z_j^{[m,\alpha]} &= X_s^2\partial_{x_j}-iy_j(\mathbb{E}+n-\alpha)(2\mathbb{E}+2n-1-2\alpha)\nonumber\\
& -iX_sq_j(2\mathbb{E}+2n-1-2\alpha)\ .
\end{align}
For the special case $\alpha = 0$, one has that $Z_j^{[m,0]} \in \mathsf{Hom}(\mathcal{M}_{\ell}^s,\mathcal{M}_{\ell+1}^s)$ {(see \cite{DBHS})}. This means that it can be seen as a raising operator in $\mathsf{End}(\ker D_s)$.  
\end{defn}
\begin{rem}
{ Recall that in the orthogonal case we started with the $\alpha$-deformed Klein inversion $\mathcal{I}_{\alpha}^{(m)}$ and constructed a raising operator $R^{[m,\alpha]}$. A first guess in order to generalise this approach into the symplectic framework is to find the symplectic analogue of the Klein inversion, which does not exist up to our knowledge he reason for this is simple: in the orthogonal case, the norm squared (which lies at the basis of the inversion) is a polynomial invariant of degree 2. In the symplectic case, no such invariant exists (the space of polynomials of fixed degree defines an irreducible representation for the regular action).}
\end{rem}
\begin{rem}
Note that there are $2n$ of these $\alpha$-deformed raising operators in total, which may seem strange as they are labelled by an index $1 \leq j \leq n$. There is however a second family of operators, defined by means of 
\begin{align}
Z_{j + n}^{[m,\alpha]} &= X_s^2\partial_{y_j}+ix_j(\mathbb{E}+n-\alpha)(2\mathbb{E}+2n-1-2\alpha)\nonumber\\
& -iX_s\partial_{q_j}(2\mathbb{E}+2n-1-2\alpha)\ .
\end{align}
It is sufficient to work with one member of this complete family, as one can easily verify (using direct calculations) that 
\[ [Y_{jj},Z_j^{[m,\alpha]}] = \big[-x_j\partial_{y_j} - \frac{i}{2}\partial_{q_j}^2,Z_j^{[m,\alpha]}\big] = Z_{j + n}^{[m,\alpha]}\ . \]
This was to be expected of course, as our (deformed) raising operators span a copy of the fundamental vector representation for $\sym(2n)$, and the relation above is essentially the action of the operator $Y_{jj}$, see (\ref{realisaties}). For this reason, we can (without loss of generality), focus on the operator from the definition above. 
\end{rem}
\noindent
Just as in the orthogonal case, the deformed raising operators have a special property when $\alpha = k \in \Z^+$ (compared to the orthogonal case, we have opted to absorb the minus sign into the definition). This is the symplectic version of item $(ii)$ in the wish list from Section 2. Before we prove this, we first mention the following technical lemma (each of these statements can easily proved by induction on the parameter $a$): 
\begin{lem}\label{comm}
For each $a \in \Z^+_0$, we have:
\begin{align*}
[q_j,X_s^a] &= -ay_jX_s^{a-1}\\
[\partial_{x_j},X_s^a] &= aX_s^{a-1}iq_j-\frac{a(a-1)}{2}X_s^{a-2}iy_j \\
[\partial_{y_j},X_s^a] &= aX_s^{a-1}\partial_{q_j}+\frac{a(a-1)}{2}X_s^{a-2}ix_j \ .
\end{align*}
\end{lem}
\noindent
We then claim the following: 
\begin{thm}\label{sfueter}
Let $m=2n$ and let $k\in \mathbb{Z}^+$ be a deformation parameter. One then has that 
\begin{align} 
& Z_j^{[m,k]} \in \mathsf{End}\big(\ker(D_s^{2k+1})\big)\ .
\end{align}
\end{thm}
\begin{proof}
In order to prove this, we will first work towards a relation of the form 		 
\[ Z_j^{[m,k]}X_s^a\ker D_s=X_s^a\left(Z_j^{[m,0]}+\Psi_j\right)\ker D_s \qquad (a \in \Z^+)\ , \]
with $\Psi_j$ an operator acting on polynomial symplectic spinors whose explicit definition can be found below. Using the previous lemma, we first of all have that 
\begin{align}\label{term1}
X_s^2\partial _{x_j}X_s^a &= X_s^2\left(X_s^a\partial_{x_j}+[\partial_{x_j},X_s^a]\right)\nonumber \\
&= X_s^{a}\left(X_s^{2}\partial_{x_j}+aX_siq_j-\frac{a(a-1)}{2}iy_j\right)\ .
\end{align}
For the second term in $Z_j^{[m,k]}$, we quite easily get that 
\begin{align}\label{term2}
& -iy_j(\mE + n - k)(2\mE + 2n - 1 - 2k)X_s^a \nonumber\\
&= -iX_s^ay_j(\mE + n + a - k)(2\mE + 2n - 1 + 2(a- k))\ ,
\end{align}
since $[y_j,X_s] = 0$ and $[\mE,X_s] = X_s$. For the third term in $Z_j^{[m,k]}$ we have
\begin{align*}
-iX_sq_j(2\mathbb{E}+2n-2k-1)X_s^a &= -iX_sq_jX_s^a(2\mathbb{E}+2n -1 + 2(a-k))\ .
\end{align*}
Again making use of the previous lemma to swap $q_j$ and $X_s^a$, we arrive at
\begin{align}\label{term3}
& -iX_sq_j(2\mathbb{E}+2n-2k-1)X_s^a \nonumber \\
&= -iX_s\big(X_s^aq_j - ay_jX_s^{a-1}\big)(2\mathbb{E}+2n-1+2(a-k))\ .
\end{align}
Adding expressions (\ref{term1}), (\ref{term2}) and (\ref{term3}), a smart rearrangement shows that 
\begin{align*}
Z_j^{[m,k]}X_s^a &= X_s^a\big(Z_j^{[m,0]} + \Psi_j\big)\ ,
\end{align*}
where the rest operator $\Psi_j$ is given by the expression 
\[ \Psi_j := -iy_j\left(2(a-2k)\mathbb{E}+\gamma(n,k,a)\right)-iX_sq_j(a-2k)\ , \]
and where the numerical constant $\gamma(n,k,a)$ given by 
\begin{align*}
\gamma(m,k,a) &= k-2ak+2k^2+2an-4kn+\frac{a(a-1)}{2}\ .
\end{align*}
We know that $Z_j^{[m,0]}$ preserves the kernel of $D_s$, but for the operator $\Psi_j$ things are slightly more complicated. Indeed, one has that 
\begin{align}\label{y_Xq_mapping}
y_j &:\ker D_s \to \ker D_s \oplus X_s\ker D_s\oplus X_s^2\ker D_s\nonumber\\ 
X_sq_j &:\ker D_s\to X_s\ker D_s\oplus X_s^2\ker D_s\ ,
\end{align}
which means that $\Psi_j$ does not preserve the kernel of the operator $D_s^{2k+1}$. As long as $a \leq 2k-2$, we can conclude from the analysis above that  
\[ Z_j^{[m,k]}X_s^a \ker D_s = X_s^a\big(Z_j^{[m,0]} + \Psi_j\big)\ker D_s \subset \bigoplus_{p = 0}^2 X_s^{a + p}\ker D_s\ ,  \]
which means that these summands still sit inside $\ker(D_s^{2k+1})$. However, for the remaining indices $a=2k$ and $a=2k-1$ it is not clear at this point why one has that 
\[ Z_j^{[m,k]}\big(X_s^{2k-1} \ker D_s \oplus X_s^{2k} \ker D_s\big) \subset \bigoplus_{p = 0}^{2k}X_s^p\ker D_s\ , \]
as the operator identities in (\ref{y_Xq_mapping}) seem to indicate that one may expect a contribution of the form $X_s^p \ker D_s$ with $p > 2k$. In order to show that this is not the case, we first of all note that $\Psi_j = 0$ for $a = 2k$. To verify that also the final case $a = 2k-1$ poses no problems, we will in fact prove that the following property holds: 
\begin{align*}
\Psi_j : \ker D_s \to \ker D_s \oplus X_s\ker D_s.
\end{align*}
To do so, it suffices to prove that $D_s^2(\Psi_j \ker D_s)=0$. Plugging in $a=2k-1$ into our definition for $\Psi_j$, one finds $\Psi_j = iy_j(2\mathbb{E}+2n-1) + iX_sq_j$. As the action of $D_s\Psi_j$ on $\ker D_s$ reduces to the commutator
\begin{align*}
[D_s,y_j(2\mathbb{E}+2n-1)+X_sq_j] &= iq_j(\mathbb{E}+n-1)-X_s\partial_{x_j}\ ,
\end{align*}
we are now left to check whether the right-hand side commutes with $D_s$ (when acting on $\ker D_s$):
	\begin{align*}
		[D_s,iq_j(\mathbb{E}+m-1)-X_s\partial_{x_j}]&=[D_s,iq_j](\mathbb{E}+n-1)-[D_s,X_s]\partial_{x_j}\\
		&= -\partial_{x_j}(\mathbb{E}+n-1)+i(\mathbb{E}+n)\partial_{x_j}\ ,
	\end{align*}
which is indeed zero. Together, this finishes the proof. 
\end{proof}
\noindent
In order to get a feel for what this theorem tells us, let us have a look at the scheme below: 
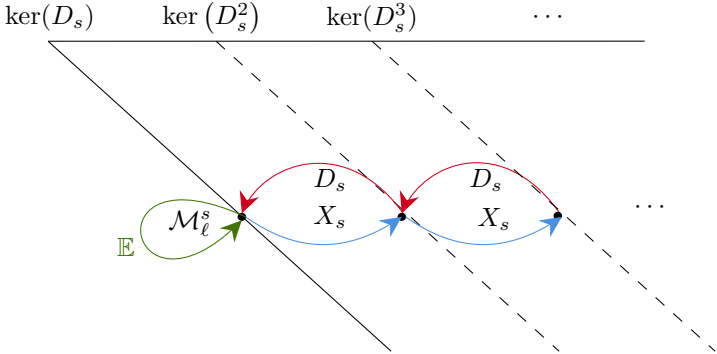
\begin{center}
	\begin{tikzpicture}[x=0.75pt,y=0.75pt,yscale=-1,xscale=1]
	
	\draw    (40,28) -- (338.3,27.96) ;

	\draw    (40,28) -- (211.54,184.69) ;

	\draw [color={rgb, 255:red, 0; green, 0; blue, 0 }  ,draw opacity=1 ] [dash pattern={on 4.5pt off 4.5pt}]  (124,28) -- (295.54,184.69) ;

	\draw  [fill={rgb, 255:red, 0; green, 0; blue, 0 }  ,fill opacity=1 ] (135,116.75) .. controls (135,115.78) and (135.84,115) .. (136.88,115) .. controls (137.91,115) and (138.75,115.78) .. (138.75,116.75) .. controls (138.75,117.72) and (137.91,118.5) .. (136.88,118.5) .. controls (135.84,118.5) and (135,117.72) .. (135,116.75) -- cycle ;
	\draw [color={rgb, 255:red, 65; green, 117; blue, 5 }  ,draw opacity=1 ]   (135.16,120.23) .. controls (79,175.83) and (60.18,80.14) .. (138.75,116.75) ;
	
	\draw [shift={(136.88,118.5)}, rotate = 134.12] [fill={rgb, 255:red, 65; green, 117; blue, 5 }  ,fill opacity=1 ][line width=0.75]  [draw opacity=0] (10.72,-5.15) -- (0,0) -- (10.72,5.15) -- (7.12,0) -- cycle    ;
	\draw [color={rgb, 255:red, 208; green, 2; blue, 27 }  ,draw opacity=1 ]   (137.66,113.09) .. controls (148.64,88.57) and (190.23,74.42) .. (218.75,116.75) ;
	
	\draw [shift={(136.88,115)}, rotate = 290.56] [fill={rgb, 255:red, 208; green, 2; blue, 27 }  ,fill opacity=1 ][line width=0.75]  [draw opacity=0] (10.72,-5.15) -- (0,0) -- (10.72,5.15) -- (7.12,0) -- cycle    ;
	\draw  [fill={rgb, 255:red, 0; green, 0; blue, 0 }  ,fill opacity=1 ] (215,116.75) .. controls (215,115.78) and (215.84,115) .. (216.88,115) .. controls (217.91,115) and (218.75,115.78) .. (218.75,116.75) .. controls (218.75,117.72) and (217.91,118.5) .. (216.88,118.5) .. controls (215.84,118.5) and (215,117.72) .. (215,116.75) -- cycle ;
	\draw [color={rgb, 255:red, 208; green, 2; blue, 27 }  ,draw opacity=1 ]   (217.66,113.09) .. controls (228.54,88.55) and (268.23,73.92) .. (296.75,116.25) ;
	
	\draw [shift={(216.88,115)}, rotate = 290.56] [fill={rgb, 255:red, 208; green, 2; blue, 27 }  ,fill opacity=1 ][line width=0.75]  [draw opacity=0] (10.72,-5.15) -- (0,0) -- (10.72,5.15) -- (7.12,0) -- cycle    ;
	\draw  [fill={rgb, 255:red, 0; green, 0; blue, 0 }  ,fill opacity=1 ] (293,116.25) .. controls (293,115.28) and (293.84,114.5) .. (294.88,114.5) .. controls (295.91,114.5) and (296.75,115.28) .. (296.75,116.25) .. controls (296.75,117.22) and (295.91,118) .. (294.88,118) .. controls (293.84,118) and (293,117.22) .. (293,116.25) -- cycle ;
	\draw [color={rgb, 255:red, 74; green, 144; blue, 226 }  ,draw opacity=1 ]   (138.75,116.75) .. controls (164.4,135.12) and (188.93,135.49) .. (215.26,117.85) ;
	\draw [shift={(216.88,116.75)}, rotate = 505.17] [fill={rgb, 255:red, 74; green, 144; blue, 226 }  ,fill opacity=1 ][line width=0.75]  [draw opacity=0] (10.72,-5.15) -- (0,0) -- (10.72,5.15) -- (7.12,0) -- cycle    ;
	
	\draw [color={rgb, 255:red, 74; green, 144; blue, 226 }  ,draw opacity=1 ]   (218.75,116.75) .. controls (244.4,135.12) and (268.93,135.49) .. (295.26,117.85) ;
	\draw [shift={(296.88,116.75)}, rotate = 505.17] [fill={rgb, 255:red, 74; green, 144; blue, 226 }  ,fill opacity=1 ][line width=0.75]  [draw opacity=0] (10.72,-5.15) -- (0,0) -- (10.72,5.15) -- (7.12,0) -- cycle    ;
	
	\draw [color={rgb, 255:red, 0; green, 0; blue, 0 }  ,draw opacity=1 ] [dash pattern={on 4.5pt off 4.5pt}]  (202,28) -- (373.54,184.69) ;

	\draw (41,16) node   {$\ker(D_s)$};
	\draw (122,16) node   {$\ker\left(D_s^2\right)$};
	\draw (202,16) node   {$\ker(D_s^3)$};
	\draw (111,120) node   {$\mathcal{M}_{\ell}^s$};
	\draw (342,112) node   {$\cdots $};
	\draw (181,117) node   {$X_s$};
	\draw (259,98) node  {$D_s$};
	\draw (291,16) node   {$\cdots $};
	\draw (181,98) node {$D_s$};
	\draw (263,118) node  {$X_s$};
	\draw (79,132) node [color={rgb, 255:red, 65; green, 117; blue, 5 }  ,opacity=1 ]  {$\mathbb{E}$};
	\end{tikzpicture}
\end{center}
This infinite triangle should be seen as a visual representation for the space $\mathsf{Pol}(\R^{2n}) \otimes \mcS(\R^{n})$ of polynomial symplectic spinors, where for instance the left edge of the triangle contains all the spaces $\mathcal{M}_\ell^s$ (starting at the vertex for $\ell = 0$), and where each horizontal line represents the ${\spl(2)}$-module generated by the operators $X_s$ and $D_s$. The $k$'th dashed line, parallel with the left edge (which counts as the zero'th line), can then be seen as the quotient space $\ker(D_s^{p+1})/\ker(D_s^p)$. So what we have just proved is that for $k \in \Z^+$, the infinite strip bounded by the $k$'th and zero'th dashed line is preserved by the action of the $k$-deformed raising operator $Z_j^{[m,k]}$. \\
\\
In order to obtain a symplectic Fueter theorem, we now need to compare the $k$-deformed raising operator with the undeformed raising operator in the lowest dimension, which is $m = 2n = 2$ (this is item $(i)$ from the wish list from {S}ection 2). Using Definition \ref{raisingsymplectic}, it is clear that this operator is given by 
\[ Z_1^{[2,0]} = X_s^2\partial_{x_1}-iy_1(\mathbb{E}+1)(2\mathbb{E}+1)-iX_sq_1(2\mathbb{E}+1)\ . \]
Note that we have put the subscript $j = 1$, as there is no choice here. If we now compare this operator with the $k$-deformed raising operator in $m = 2n$ dimensions, it is clear that they are formally equal provided
\begin{align}\label{keuze}
2n-2k-1=1\Longleftrightarrow k=n-1.
\end{align}
\begin{rem}
Recall that this is a {\em formal} equality only, in the sense that the exact meaning of the symbol $X_s$ differs (it is a summation involving each of the $2n$ variables). This is not different from the classical (orthogonal) case, where one had to work with the symbol $\overline{e}_1\ux$ as a placeholder for both $m = 2$ (which then reduces to the complex variable $z$) and the general case.  
\end{rem}
\noindent
As mentioned in item $(iii)$ from the wish list at the end of Section 2, the true power of the Fueter theorem is that one can perform this formal substitution at the end, provided one supplements it with an appropriate action of the Dirac operator. In the symplectic case, a similar trick will be used to construct symplectic monogenics. We formulate this as the first (weak) version of our new Fueter theorem: 
\begin{thm}[Weak Fueter]
If $M_\ell(X_s;q_1)\psi(q_1) \in \mathsf{Pol}(\R^2) \otimes \mcS(\R)$ is defined for all $\ell \in \Z^+$ by means of 
\[ M_\ell(X_s;q_1)\psi(q_1) := \big(Z_1^{[2,0]}\big)^\ell e^{-\frac{1}{2}q_1^2} \ , \]
a formal substitution of the variable $X_s$ and $q_1 {\mapsto} \boldsymbol{q}{=}(q_1,\ldots,q_n)$ leads to 
\[ D_s^{2n - 2}\bigg(M_\ell(X_s;q_1)\psi(\boldsymbol{q})\bigg) \in \big(\mathsf{Pol}(\R^{2n}) \otimes \mcS(\R^n)\big) \cap \ker D_s\ . \]
\end{thm}
\begin{proof}
The proof for this theorem basically consists of two parts. On a purely formal level, the conclusion follows from theorem \ref{sfueter} and the observation that (formally) one has that the undeformed raising operator in dimension $m = 2$ coincides with the $k$-deformed raising operator in $m = 2n$ for $k = n-1$. Indeed, an $\ell$-fold action of the latter belongs to $\ker(D_s^{2n-1})$, so
\[ D_s^{2n-2}\bigg(\big(Z_1^{[m,n-1]}\big)^\ell\psi\bigg) \in \ker D_s\ , \]
with $\psi \in \mcS(\R^n)$ a symplectic spinor. The second part of the proof is there to make sure that this formal identification is indeed possible. On the level of spinors, this is not difficult: for $m = 2$, we let our (undeformed) raising operator act on the special element $\psi(q_1) := \exp(-\frac{1}{2}q_1^2)$. Switching from $m = 2$ to $m = 2n$, which means that one goes from $q_1 \in \R$ to $\boldsymbol{q} \in \R^n$, then boils down to the multiplication with a suitable element of $\mcS(\R^{n-1})$, as 
\[ \psi(q_1) {\mapsto} \psi(\boldsymbol{q}) = \psi(q_1)\prod_{j = 2}^n\psi(q_j)\ . \]
As the operator $\partial_{q_1}$ hidden in the raising operator $Z_1^{[2,0]}$ only acts on $\psi(q_1)$, the transition from $q_1$ to $\boldsymbol{q}$ is perfectly defined. Showing that also the formal substitution 
\[ y_1\partial_{q_1} + ix_1q_1 \mapsto \sum_{j=1}^n \big(y_j\partial_{q_j}+ix_jq_j\big) \]
makes sense requires more work, and will be done in a series of calculations below. 
\end{proof}
\begin{rem}
It might be useful to explain in more detail what exactly is missing at this point. In the classical orthogonal case, the action of the raising operator (in dimension $m = 2$) simply generates powers $(\overline{e}_1 \ux)^\ell$, which are then recognised as holomorphic powers. This makes it almost trivial to see that the formal substitution for $\ux$ (from $m = 2$ to an arbitrary even dimension $m$) is well-defined. In the symplectic case, it is by no means clear at this point that the action of the raising operator (for $m = 2$) leads to a polynomial symplectic spinor in which $X_s$ appears as a formal variable, which can then again be replaced by its counterpart in $m = 2n$ dimensions. For that purpose, we first have to derive a more explicit expression. 
\end{rem}
\noindent
In the orthogonal case, the repeated action of $R^{[2,0]}$ on a spinor in $\mS \subset \C_m$ (or simply $1 \in \C_m$) gave rise to holomorphic powers $(x_1 + \overline{e}_1e_2 x_2)^\ell$. An obvious question is the following: what is the symplectic version of this? So in a sense, we are about to derive the symplectic version of the holomorphic powers. Note that in the classical case, one could do calculations in the Clifford algebra $\C_m$ and switch to spinors at the end by means of a multiplication with $I$ (cfr. the introduction on page 2). In the symplectic case, something similar will be true, but one has to take into account that the symplectic Clifford `variable' $X_s$ contains derivatives with respect to the variables $\boldsymbol{q} \in \R^n$ which will act on the chosen symplectic spinor $\psi \in \mcS(\R^n)$. The equivalent of the complex variable $z \in \C$ is thus given by the action of $Z^{[2,0]}$ on a spinor $\psi \in \mcS(\R)$: 
\begin{align*}
\left(X_s^2\partial _{x}-iy(\mathbb{E}+1)(2\mathbb{E}+1)-iX_sq(2\mathbb{E}+1)\right){\psi}=-i(y+X_sq){\psi} ,
\end{align*}
with $X_s = ixq + y\partial_q$. Note that we have switched to the notation $(x,y;q)$ instead of $(x_1,y_1;q_1)$ to enlighten the notation. Our goal is now to explicitly find a closed form for the $p$-fold action of
this raising operator. The following theorem shows that there exists a closed expression in terms of the `variables' $y$, $X_s$ and $q$ (recall that $X_s$ is an operator). 
\begin{thm} 
The repeated action of the (undeformed) raising operator on a symplectic spinor $\psi \in  \mathcal{S}(\mathbb{R})$ is given by 
\begin{align*}
\left(X_s^2\partial _{x}-iy(\mathbb{E}+1)(2\mathbb{E}+1)-iX_sq(2\mathbb{E}+1)\right)^p{\psi}=\left(\sum_{j=0}^{p}\gamma_j^{(p)}y^jX_s^{p-j}q^{p-j}\right){\psi},
\end{align*}
where the coefficients $\gamma_j^{(p)}$ are given by: 
\begin{align}\label{coefficientengamma}
{ 	\gamma_j^{(p)}=(-i)^p p!(2j-1)!!  {p + j \choose 2j}.}
\end{align}
\end{thm}
\begin{proof}
The proof follows from an easy induction argument over $p$. Suppose that the formula is valid for $p$, and consider the expression
\begin{align}
\left(X_s^2\partial	 _{x}-iy(\mathbb{E}+1)(2\mathbb{E}+1)-iX_sq(2\mathbb{E}+1)\right)\left[\sum_{j=0}^{p}\gamma_j^{(p)}y^jX_s^{p-j}q^{p-j}\right].\label{actie}
\end{align}
Invoking the fact that $\partial_x \psi = 0$ (which explains why the term for $j = p$ is no longer present), we first of all get that
\begin{align*}
X_s^2\partial_x\sum_{j=0}^{p-1}\gamma_j^{(p)}y^j X_s^{p-j}q^{p-j} &= 
	X_s^2\sum_{j=0}^{p-1}\gamma_j^{(p)}y^j[\partial_x, X_s^{p-j}]q^{p-j} \ .
\end{align*}
Using Lemma \ref{comm}, this can then be written as 
\begin{align*}
i\sum_{j=0}^{p-1}\gamma_j^{(p)}\left(	(p-j)X_s^{p-j+1}q^{p-j+1}y^j-\frac{(p-j)(p-j-1)}{2} X_s^{p-j}q^{p-j}y^{j+1}	\right).
\end{align*}
As the $p$-fold action of the raising operator on $\psi$ is homogeneous of degree $p$, the second term of the raising operator gives the following contribution: 
\begin{align*}
-i(p+1)(2p+1) \left(\sum_{j=0}^{p}\gamma_j^{(p)}y^{j+1}X_s^{p-j}q^{p-j}\right)\ .
\end{align*}
As for the term $iX_sq(2\mathbb{E}+1)$, we get a contribution of the form 
\begin{align*}
& i(2p+1) \sum_{j=0}^{p}\gamma_j^{(p)}y^jX_s\left(X_s^{p-j}q+[q,X_s^{p-j}]\right)q^{p-j}\\
=& i(2p+1) \sum_{j=0}^{p}\gamma_j^{(p)}\left(	X_s^{p-j+1}q^{p-j+1}y^j-(p-j)X_s^{p-j}q^{p-j}y^{j+1}	\right)\ ,
\end{align*}
where we have again used Lemma \ref{comm}. Putting all these terms together and rearranging the summations then indeed leads to the desired result. 
\end{proof}
\noindent
Now that we have found the coefficients $\gamma_j^{(p)}$, we can try to rewrite them in such a way that our polynomial symplectic spinors can be expressed in terms of known special functions. In view of the fact that 
\begin{align*}
(2j-1)!!  {p + j \choose 2j} &= \frac{(2j)!}{2^j j!}\frac{(p+j)!}{(2j)!(p-j)!} =\frac{(p+j)!}{(p-j)!j!}\frac{1}{2^j}\ ,
\end{align*}
one may recognise the coefficients of the so-called Bessel polynomials (see \cite{FK}). These are given by 
\begin{align}
\beta_p(x) := \sum_{j = 0}^p \frac{(p+j)!}{(p-j)!j!}\left(\frac{x}{2}\right)^j \qquad (p \in \Z^+)\ .
\end{align}
So in a sense, whereas the Gegenbauer polynomials appear naturally in the classical framework, one ends up with Bessel polynomials in the symplectic framework. Bringing everything together, we have thus proved the following: 
\begin{cor}
The repeated action of the undeformed raising operator $Z^{[2,0]}$ on a symplectic spinor $\psi \in \mathcal{S}(\mathbb{R})$ is given by 
\begin{align*}
\big(Z^{[2,0]}\big)^p{\psi} &= (-i)^p p!\left(\sum_{j=0}^{p}B(p,j)y^jX_s^{p-j}q^{p-j}\right){\psi},
\end{align*}
where $B(p,j)$ stands for the coefficients appearing in the Bessel polynomials: 
\begin{align*}
B(p,j)=\frac{(p+j)!}{(p-j)!j!}\frac{1}{2^j}.
\end{align*}
\end{cor}
\begin{rem}
Note that if one replaces $(y,q) \mapsto (y_j,q_j)$ and the symplectic `variable' $X_s$ by its counterpart in $m = 2n$ dimensions, then the expression above can also be used for the repeated action of the deformed operator $Z_j^{[m,n-1]}$ acting on an arbitrary symplectic spinor $\psi \in \mcS(\R^n)$. 
\end{rem}
\noindent
The remark above, together with the corollary, enables us to give our final version of the symplectic Fueter theorem (the `strong' version, in which the substitution is seen to make sense). Note that we hereby make a particular choice for the symplectic spinor, with $\psi(q) = \exp(-\frac{1}{2}q^2)$. 
\begin{thm}[Symplectic Fueter Theorem]
If we define the symplectic spinor $M_\ell(X_s;q_1)\psi(q_1)$ as
\begin{align}
M_\ell(X_s;q_1)\psi(q_1) &:= (-i)^\ell \ell!\left(\sum_{j=0}^{\ell}B(\ell,j)y^jX_s^{\ell-j}q^{\ell-j}\right)[e^{-\frac{1}{2}q_1^2}]\ ,
\end{align}
a formal substitution leads to 
\begin{align}
D_s^{2n-2}\bigg(M_\ell(X_s;q_1)\psi(\boldsymbol{q})\bigg) \in \left(\mathsf{Pol}(\R^{2n})\otimes\mcS(\R^n)\right) \cap \ker D_s\ .
\end{align}
We hereby stress that in this last formula, $X_s$ denotes the symplectic (formal) variable in $m = 2n$ dimensions. 
\end{thm}
\noindent
To make the analogy with the classical case even more apparent, we will use some tricks to reformulate this theorem slightly. To do so, we first introduce the following notation: 
\begin{align}
(X_s q_j)^{\ast a} := X_s^a q_j^a \qquad (a \in \Z)\ .
\end{align}
Note that this is a formal notation only, because we allow $a \in \Z$ to be a negative integer here. Although this does not make sense (recall that $X_s$ is an operator), these negative powers will never really occur (see below). We have choosen for the notation `$\ast$' here, but one may think of this as the so-called normal ordering operator appearing in the physics literature. Indeed, identifying $(X_s,q_j)$ with $(b^\dagger,b)$ and switching to the notation used in physics (with e.g. $:bb^\dagger:\ = b^\dagger b$), one clearly has that 
\[ (X_s q_j)^{\ast a} =\ :X_s q_j X_s q_j \ldots X_s q_j:\ = X_s^aq_j^a\ . \]
The $\ast$-multiplication (or normal ordering) of two operators of the form $(X_s q_j)^{\ast a}$ is then defined by means of 
\[ (X_s q_j)^{\ast a} \ast (X_s q_j)^{\ast b} := (X_s q_j)^{\ast (a + b)} \qquad (a,b \in \Z)\ .  \]
In terms of this notation, the symplectic Fueter theorem can be formulated as follows: 
\begin{thm}
For all $\ell \in \Z^+$, one has that 
\begin{align*}
D_s^{2n-2}\bigg((X_sq_j)^{\ast \ell}\ast\beta_\ell\left(t\right)\bigg) \in \ker D_s\ ,
\end{align*}
where the powers of the `variable' $t$ are defined by means of 
\[ t^a = \left(\frac{y}{X_s q_j}\right)^a := y^a(X_sq_j)^{\ast(-a)} = y^aX_s^{-a}q_j^{-a}\ . \]
\end{thm}
\noindent
This may look ill-defined, but as mentioned above these negative powers of $(X_s q_j)$ never really occur, in view of the factor $(X_sq_j)^{\ast \ell}$ appearing in front of the Bessel polynomial. 


\nocite{*}
\bibliographystyle{alpha}

\begin{thebibliography}{200}
	
	\bibitem{BDS} F. Brackx, R. Delanghe and F. Sommen, {Clifford Analysis}, Research Notes in Mathematics 76, Pitman, London (1982)
	
	\bibitem{CH1} M. Cahen, S. Gutt and J. Rawnsley, {Symplectic Dirac operators and $Mp^c$-structures}, Gen. Relativ. Gravit. {\bf 43}, 3593-3617 (2011)
	
	\bibitem{CH2} M. Cahen and S. Gutt, {Spin$^c$, Mp$^c$ and symplectic Dirac operators}, in
	Geometric Methods in Physics, XXXI  Workshop 2012, ed. by P. Kielanowski et al., Trends in Mathematics, Birkh\"auser, 13-28 (2013)
	
	\bibitem{CH3} M. Cahen, S. Gutt, L. La Fuente Gravy and J. Rawnsley, {On Mp$^c$-structures and Symplectic Dirac Operators}, J. Geom. Phys., 434-466 (2014)
	
	\bibitem{CSS} { F. Colombo, I. Sabadini and F. Sommen, {The inverse Fueter mapping theorem},  Communications On Pure And Applied Analysis,  \textbf{10} No. 4, 1165-1181 (2011)}
	
	\bibitem{DBHS} H. De Bie, M. Hol{\'\i}kov{\'a} and P. Somberg, {Basic aspects of symplectic Clifford analysis for the symplectic Dirac operator}, Adv. Appl, Cliff. Alg. {\bf 27} No. 2, 1103-1132 (2017)
	
	\bibitem{DBSS} H. De Bie, P. Somberg and V. Sou\v{c}ek, {The metaplectic Howe duality and polynomial solutions for the symplectic Dirac operator}, J. of Geom. and Phys. {\bf 75}, 120-128 (2014)
	
	\bibitem{DSS} R. Delanghe, F. Sommen and V. Sou\v{c}ek, {Clifford analysis and spinor valued functions}, Kluwer Academic Publishers, Dordrecht (1992).
	
	\bibitem{ESVL} D. Eelbode, V. Sou\v{c}ek and P. Van Lancker, {Gegenbauer polynomials and the Fueter theorem}, {  Complex Variables And Elliptic Equations \textbf{59} No. 6, 826-840 (2014)}
	
	\bibitem{ESVL2} D. Eelbode, V. Sou\v{c}ek and P. Van Lancker, {The Fueter theorem by representation theory}, AIP Conference Proceedings {\bf 1479} No. 1, 340--343 (2012)
	
	\bibitem{FK} O. Frink and H. L. Krall, {A New Class of Orthogonal Polynomials: The Bessel Polynomials}, Trans. Amer. Math. Soc. {\bf 65} No. 1, 100-115 (1948)
	
	\bibitem{F} R. Fueter, {Die Funktionentheorie der Differentialgleichungen $\Delta u = 0$ und $\Delta\Delta u = 0$ mit vier reellen
		Variablen}, Comm. Math. Helv. 7, 307-330 (1935)
	
	\bibitem{GM} J. Gilbert and M.A.M. Murray, {Clifford algebras and Dirac operators in harmonic analysis}, Cambridge University Press, Cambridge (1991)
	
	\bibitem{Ha} A. Haydys, {Generalized Seiberg–Witten equations and hyperK\"ahler geometry}, PhD. thesis, Georg-August University of G\"ottingen (2006) 
	
	\bibitem{HKS} M. Hol{\'\i}kov{\'a}, L. K\v{r}i\v{z}ka and P. Somberg, {$\widetilde{\mathrm{SL}}(3,\mathbb{R})$ and the symplectic Dirac operator}, Arch. Math. {\bf 52}, 2041-2052 (2016)
	
	\bibitem{HNS} S. Hohloch, G. Noetzel and D. A. Salamon, {Hypercontact structures and Floer homology}, Geom. Topol. {\bf 13} No. 5, 2543--2617 (2009)
	
	\bibitem{KQS} K.I. Kou, T. Qian and F. Sommen, {Generalizations of Fueter's Theorem}, Meth. Appl. Anal. 9 No. 2, 273-290 (2002)
	
	\bibitem{Ni} A. Nita, {Essential Self-Adjointness of the Symplectic Dirac Operators}, Mathematics Graduate Theses \& Dissertations {\bf 45} (2016)
	
	\bibitem{PPS} D. Pe\~na Pe\~na and F. Sommen, {Fueter's theorem: the saga continues}, J. Math. Anal. Appl. 365 No.1, 29-35 (2012)
	
	\bibitem{Q} T. Qian, {Generalization of Fueter's result to $\R^{m+1}$}, Rend. Mat. Acc. Lincei 9, 111-117 (1997)
	
	\bibitem{Sc} M. Sce, {Osservazioni sulle serie di potenze nei moduli quadratici}, Atti Acc. Lincei Rend. Fisica, \textbf{23}, 220–225 (1957)
	
	\bibitem{S} F. Sommen, {On a generalization of Fueter's theorem}, Z. Anal. Anw. 19,  899-902 (2000)
	
	\bibitem{W} T. Walpuski, {A compactness theorem for Fueter sections}, Commentarii Mathematici Helvetici, Volume 92, Issue 4, 751--776 (2017)
\end{thebibliography}

\end{document}